\documentclass[11pt]{amsart}
\usepackage{amssymb}
\usepackage{amsmath}
\usepackage{amsrefs}
\usepackage[colorlinks=true]{hyperref}
\usepackage{cleveref}
\usepackage{amscd}
\usepackage{amsthm}
\usepackage{mathtools}
\usepackage{tikz-cd}
\usepackage{bm}
\usepackage{mathrsfs}
\usepackage{enumerate}
\usepackage{enumitem}

\usepackage{xcolor}
\hypersetup{
    colorlinks,
    linkcolor={red!50!black},
    citecolor={green!50!black},
    urlcolor={blue!80!black}
}
\allowdisplaybreaks

\newtheorem{theorem}{Theorem}[section]
\newtheorem{lemma}[theorem]{Lemma}
\newtheorem{proposition}[theorem]{Proposition}
\newtheorem{corollary}[theorem]{Corollary}

\newtheorem{conjecture}[theorem]{Conjecture}
\newtheorem{question}[theorem]{Question}

\newtheorem*{assumption-no-number}{Assumption}
\newtheorem*{corollary*}{Corollary}

\theoremstyle{definition}
\newtheorem{definition}[theorem]{Definition}
\newtheorem{example}[theorem]{Example}
\newtheorem{notation}[theorem]{Notation}

\newtheorem*{notation*}{Notation}

\theoremstyle{remark}
\newtheorem{remark}[theorem]{Remark}

\numberwithin{equation}{section}

\newcommand{\R}{\mathbb{R}}

\newcommand{\rmk}{\begin{remark}}
\newcommand{\ermk}{\end{remark}}
\newcommand{\cor}{\begin{corollary}}
\newcommand{\ecor}{\end{corollary}}
\newcommand{\eq}{\begin{equation}}
\newcommand{\eeq}{\end{equation}}
\newcommand{\eqs}{\begin{equation*}}
\newcommand{\eeqs}{\end{equation*}}
\newcommand{\prop}{\begin{proposition}}
\newcommand{\eprop}{\end{proposition}}
\newcommand{\thm}{\begin{theorem}}
\newcommand{\ethm}{\end{theorem}}
\newcommand{\conj}{\begin{conjecture}}
\newcommand{\econj}{\end{conjecture}}
\newcommand{\lem}{\begin{lemma}}
\newcommand{\elem}{\end{lemma}}
\newcommand{\defi}{\begin{definition}}
\newcommand{\edefi}{\end{definition}}
\newcommand{\ex}{\begin{example}}
\newcommand{\eex}{\end{example}}
\newcommand{\alis}{\begin{align*}}
\newcommand{\ealis}{\end{align*}}
\newcommand{\pf}{\begin{proof}}
\newcommand{\epf}{\end{proof}}
\newcommand{\ali}{\begin{align}}
\newcommand{\eali}{\end{align}}
\newcommand{\qus}{\begin{question}}
\newcommand{\equs}{\end{question}}

\newcommand{\ov}{\overline}

\newcommand{\op}{\operatorname}

\newcommand{\e}{\epsilon}

\renewcommand{\ov}{\overline}

\newcommand{\Ham}{\operatorname{Ham}}
\newcommand{\Symp}{\operatorname{Symp}}
\newcommand{\CSymp}{\operatorname{\widetilde{Symp}}}

\title{A note on the cardinality of Lagrangian packings} 

\begin{document}
\author{Jo\'{e} Brendel$^*$}\thanks{$^*$Supported by the Swiss National Science Foundation Ambizione Grant PZ00P2-223460.}
\address {ETH Z\"urich, Zurich, Switzerland}
\email {joe.brendel@math.ethz.ch}
\author{Jean-Philippe Chass\'{e}$^{\dagger}$}\thanks{$^{\dagger}$Partially supported by the Swiss National Science Foundation (grant number  200021\_204107)}
\address {CRM, Montreal, Canada}
\email {jean-philippe.chasse@umontreal.ca}
\author{Laurent C\^{o}t\'{e}$^{\ddagger}$}\thanks{$^{\ddagger}$Supported by the Hausdorff Center for Mathematics (EXC-2047/1–390685813).}
\address{Universit\"at Bonn, Mathematical Institute, Bonn, Germany}
\email{lcote@math.uni-bonn.de}

\begin{abstract}
Given a symplectic manifold, can one pack uncountably many Lagrangian submanifolds in a given Hamiltonian isotopy class of this symplectic manifold? We address $C^\infty$ and $C^0$ versions of this question.
\end{abstract}

\maketitle

\section{Introduction}

\subsection{Lagrangian packing problems}
Given a symplectic manifold $M = (M, \omega)$, let $\Ham(M) \subset \op{Diff}_0(M)$ be the group of compactly-supported Hamiltonian diffeomorphisms of $M$. By definition, the elements of $\Ham(M)$ are those diffeomorphisms which occur as the time-$1$ flow of a compactly-supported Hamiltonian function $H:[0,1]\times M\to\R$. Two (closed) submanifolds $L, L' \subset M$ are said to be Hamiltonian isotopic if $L'= \phi(L)$ setwise, for some $\phi \in \op{Ham}(M)$. 

A Lagrangian packing problem in symplectic geometry usually takes the following form:  

\emph{Fix a symplectic manifold $M$ and a closed Lagrangian submanifold $L \subset M$. How many pairwise disjoint Hamiltonian isotopic copies of $L$ can you pack into $M$?}  
\begin{itemize}
	\item Sometimes this number is $1$: this happens, for example, if the Floer homology is defined and nonzero, so that $L$ is non-displaceable, e.g.\ the equator in $S^2$ or the zero section in any cotangent bundle.
	\item Sometimes this number is finite but strictly greater than $1$: this holds trivially for small contractible Lagrangians for area reasons whenever $\operatorname{dim}(M)=2$. Remarkably, Polterovich--Shelukhin exhibited  displaceable Lagrangian tori in a non-monotone $S^2 \times S^2$ with finite packing number \cite[Theorem~C]{p-s}, using asymptotic Hofer geometry. 
	\item Sometimes this number is infinite: for instance, it follows from work of Chekanov \cite{ch} that there are infinite packings of Lagrangian tori in any Darboux ball of dimension at least six. In subsequent work, Chekanov--Schlenk \cite[Theorem 1.5]{c-s} give similar examples in dimension four. See also \cite{Bre23,BreKim23} for constructions of such packings using symmetric probes. 
\end{itemize}
These results beg the following question:
\begin{question}\label{question:main}
	Do there exist uncountable Lagrangian packings?
\end{question}

As it turns out, the answer is no:
\begin{proposition}\label{proposition:intro}
	The cardinality of any Lagrangian packing is at most countable.
\end{proposition}
\begin{proof}
	Let $\mathcal{H} = C^\infty(S^1 \times M) \subset C^2(S^1 \times M)$ be endowed with the subspace topology. Let $\psi^1_H$ denote the time-$1$ flow of $H \in \mathcal{H}$ and let $d(-,-)$ be the restriction of the $C^2$-metric to $\mathcal{H}$.  
	
	Suppose there exists an uncountable packing $\{L_\alpha\}$, and choose Hamiltonians $H_\alpha \in \mathcal{H}$ so that $L_\alpha= \psi^1_{H_\alpha}(L)$. Note that, for each $\alpha$, there exists $\epsilon_\alpha>0$ such that $\psi^1_{H_\alpha}(L) \cap \phi^1_K(L)\neq \emptyset$ whenever $d(H_\alpha, K) < \e_\alpha$. Indeed, if two Hamiltonians are $C^2$-close, then their associated Hamiltonian diffeomorphism must be $C^1$-close, so that we can see one Lagrangian as a graph over the other (in a Weinstein neighbourhood). Hence the balls of size $\epsilon_\alpha/2$ around each $H_\alpha$ must all be disjoint. But $\mathcal{H}$ is second countable (being separable and metrizable), so does not admit an uncountable collection of disjoint balls.
\end{proof}

Using classical results of Laudenbach--Sikorav \cite{laudenbach1994hamiltonian}, one can in fact show that a (closed connected) half-dimensional submanifold $L \subset M$ admits an uncountable Hamiltonian packing if and only if $L$ is \emph{not} Lagrangian and its normal bundle admits a nowhere vanishing section; see Proposition~\ref{prop:nonLag_subcritical}.\footnote{Note the following amusing consequence: if $L$ admits an uncountable Hamiltonian packing, then $L$ admits a packing with the cardinality of the reals, independently of the continuum hypothesis!}

\subsection{The $C^0$ setting}
Question~\ref{question:main} can also be formulated in the context of $C^0$ symplectic geometry, where it intersects subtle questions related to flux and $C^0$-rigidity. 

The study of $C^0$ symplectic geometry was initiated by Gromov and Eliashberg’s discovery that the group of symplectomorphisms of a closed symplectic manifold is $C^0$-closed in the group of volume-preserving diffeomorphisms \cite{gromov1986,eliashberg1987}. Since then, a number of symplectic invariants and properties~---~such as a submanifold being coisotropic, the closure of $\Ham(M)$ in $\Symp(M)$, and continuity properties of spectral invariants~---~have been shown to extend in the $C^0$ topology (see e.g.\ \cite{HumiliereLeclercqSeyfaddini2015,buhovsky2015,BuhovskyHumiliereSeyfaddini2021,m-o} and the references therein). We briefly recall the relevant definitions below.

Given a symplectic manifold $M = (M, \omega)$, let $$\ov{\Ham}(M) \subset \op{Homeo}_c(M)$$ the $C^0$ closure of $\Ham(M)$ in the group of all compactly-supported homeomophisms of $M$. Elements of $\ov{\Ham}(M)$ are called Hamiltonian \emph{homeomorphisms}. 

\begin{definition}\label{definition:packing}
	Fix a smooth submanifold $\Sigma \subset M$. A \emph{$C^0$ Lagrangian packing of $\Sigma$} is a collection of pairwise disjoint smooth submanifolds $\{\Sigma_\alpha\}$ such that $\Sigma_\alpha = \phi_\alpha(\Sigma)$ for some $\phi_\alpha \in \ov{\Ham}(M, \omega)$. 
\end{definition}

We can now ask:
\begin{question}\label{question:main-c0}
	Do there exist uncountable $C^0$ Lagrangian packings (i.e.\ does there exist a (closed) Lagrangian submanifold $L \subset M$ which admits an uncountable packing in the sense of Definition~\ref{definition:packing})?
\end{question}
A naive attempt to replicate the proof of Proposition~\ref{proposition:intro} breaks down. We of course know that $\overline{\operatorname{Ham}}(M)$ is metrizable and second countable (since it sits inside the space of continuous self maps of $M$). However, given $\psi \in \overline{\operatorname{Ham}}(M)$, we \emph{do not know} whether there exists a ball $B$ centered at $\psi$ with the property that $\psi'(L) \cap \psi(L) \neq \emptyset$ whenever $\psi' \in B$. If $L$ is such that $\omega(H_2(M,L))$ is discrete, then this is proved in \cite{a-c-l-s}; the general case is essentially \cite[Conjecture E]{a-c-l-s}. Note that this rationality condition is much stronger than what appears below.

Instead, we consider a different line of argument in the spirit of \cite{a-c-l-s}, which involves an analysis of the flux morphism. This yields another elementary\footnote{Meaning both that it is equally easy and that neither proof needs $J$-holomophic curves.} proof of Proposition~\ref{proposition:intro}, and also gives the harder:\footnote{Meaning both that the argument is more difficult and that it relies on $J$-holomorphic curves (although these only enter indirectly, through results we quote).}

\begin{theorem}\label{theoreom:main-co}
	Let $M=(M, \omega)$ be a symplectic manifold. Let $L \subset M$ be a (closed connected) half-dimensional submanifold with the property that \begin{equation}\label{equation:top-assumption-intro} \langle\Gamma_{top},\iota_*(H_1(L))\rangle\cap \omega(\pi_2(M)) \subset \mathbb{R}\end{equation} is discrete. Then $L$ admits an uncountable $C^0$ packing if and only if $L$ is \emph{not} Lagrangian and its normal bundle admits a nowhere vanishing section.
\end{theorem}
Here $\Gamma_{top}$ is the \emph{topological flux group} of $M$, a notion which we review in Subsection~\ref{subsection:lag-flux}. Essentially, $\langle\Gamma_{top},\iota_*(H_1(L))\rangle$ consists of evaluations of $\omega$ on a suitable subgroup of $H_2(M)$ associated to certain tori $S^1\times S^1\to M$ such that the restriction to $\{1\}\times S^1$ is a loop in $L$. In practice, \eqref{equation:top-assumption-intro} is reasonably checkable: it obviously holds whenever $\omega(\pi_2(M))$ is discrete, or when $\iota_*:H_1(L;\R)\to H_1(M;\R)$ is zero (so e.g.\ it is enough for $M$ to be simply connected). Nevertheless, we expect that the conclusion of Theorem~\ref{theoreom:main-co} holds for all symplectic manifolds, without any additional topological condition.

\section{Preparations}

\subsection{Conventions}
All rings are understood to be commutative and unital. All manifolds are by definition boundaryless, Hausdorff and second-countable. Unless otherwise indicated, all manifolds and all maps between them are assumed to be smooth. 

\subsection{Some topological properties of the space of Lagrangians}
The purpose of this paragraph is to show that the space of compact Lagrangian embeddings into some symplectic manifold $(M,\omega)$ is second-countable and to deduce the following. 

\begin{lemma}\label{lemma:all-in-weinstein}
	Let $M= (M, \omega)$ be a connected symplectic manifold. Let $\mathcal{X}$ be a set of pairwise disjoint, compact Lagrangian submanifolds of $M$. If $\mathcal{X}$ is uncountable, then there exists $K \in \mathcal{X}$ and a Weinstein neighborhood of $K$ containing uncountably many elements of $\mathcal{X}$.
\end{lemma}

To start, let $L$ be a smooth, compact manifold, and let $M$ be a smooth manifold. Let $C^\infty(L, M)$ be the set of infinitely differentiable maps from $L$ to $M$, endowed with the Whitney topology \cite[II.\S 3]{g-g}. Recall that a basis for this topology is given by the preimages of all open subsets $U \subset J^k(L, M), 1 \leq k < \infty$ under the natural map sending a function to its $k$-th jet. The space $C^\infty(L, M)$ is metrizable and separable, hence second countable \cite[Sec.\ 1.1]{hirsch}.

Since the $J^k(L, M)$ are manifolds and hence second-countable, $C^\infty(L, M)$ is second-countable. Let $C^\infty_{emb}(L, M) \subset C^\infty(L, M)$ be the (open) subset of embeddings and let 
\begin{equation*}
	\Sigma^\infty(L, M) := C^\infty_{emb}(L, M)/ \op{Diff}(L)
\end{equation*}
be the space of embedded smooth submanifolds of $M$ which are diffeomorphic to $L$. Here $\op{Diff}(L)$ acts on the right by precomposition. It follows that $\Sigma^\infty(L, M)$ is second-countable. 

Now let $(M,\omega)$ be a symplectic manifold and denote by $\Sigma ^\infty_{Lag}(L, M) \subset \Sigma ^\infty_{emb}(L, M)$ the subset of Lagrangian embeddings of $L$ into $M$. It follows that $\Sigma ^\infty_{Lag}(L, M)$ is second-countable. 

\begin{definition}
	Let $(M, \omega)$ be a symplectic manifold and let $\iota: L \to M$ be a Lagrangian embedding. A \emph{Weinstein neighborhood} subordinate to the embedding $\iota$ is the data $(\mathcal{U}, \tilde{\iota})$ of an open subset $\mathcal{U} \subset T^*L$ containing the zero section, and a symplectic embedding $\tilde{\iota}: \mathcal{U} \to M$ extending $\iota$. 
\end{definition}

Let $\iota: L \to (M, \omega)$ be a Lagrangian embedding. Fix a Weinstein neighborhood $(\mathcal{U}, \tilde{\iota})$ and let
\begin{equation}
	\mathcal{O}_{Lag}(\mathcal{U},\tilde{\iota}) 
	\subset \Sigma^\infty_{Lag}(L, M)
\end{equation}
be the subset of those elements $K \in \Sigma^\infty_{Lag}(L, M)$ such that there exists a $1$-form $\alpha_K \in \Omega^1(L)$ on $L$ such that $\op{graph}(\alpha_K) \subset \mathcal{U}$ and $\tilde{\iota}(\op{graph}\alpha_K)= K$. Since $K$ is Lagrangian, any such $1$-form is necessarily closed, $d\alpha_K = 0$. The sets $\mathcal{O}_{Lag}(\mathcal{U},\tilde{\iota}) $ are open. We are now in a position to prove the lemma.\\

\begin{proof}[Proof of Lemma \ref{lemma:all-in-weinstein}]
Since there are countably many diffeomorphism types, we may assume all elements of $\mathcal{X}$ are diffeomorphic to some compact manifold $L$. In other words,  $\mathcal{X}$ is a subset of $\Sigma^\infty_{Lag}(L, M)$. If we endow $\mathcal{X}$ with the subspace topology, it is second-countable, since $\Sigma^\infty_{lag}(L, M)$ is.  For each $K \in \mathcal{X}$, choose a Weinstein neighborhood $(\mathcal{U}_K, \tilde{\iota}_K)$, where $\iota_K: K \hookrightarrow M$ is the tautological inclusion.  The $\mathcal{O}(\mathcal{U}_K, \tilde{\iota}_K)$ are open in $\Sigma^\infty_{lag}(L, M)$; hence their restriction to $\mathcal{X}$ is also open. Hence they form an open cover of $\mathcal{X}$. By second-countability, there exists a countable subcover.  Hence there exists some $K \in \mathcal{X}$ such that uncountably many of the elements of $\mathcal{X}$ are contained in $\mathcal{O}(\mathcal{U}_K, \tilde{\iota}_K)$. 
\end{proof}

\subsection{Notions of flux}\label{subsection:lag-flux}
There are several notions of ``flux'' in symplectic geometry which all arise from variants of the same construction. We briefly review (some of) these here.

Let $M=(M, \omega)$ be a symplectic manifold. 
Let $\{L_t\}_{t \in [0,1]}$ be a Lagrangian isotopy. 
The \emph{(Lagrangian) flux} of the isotopy is the class $\operatorname{Flux}(L_s) \in H^1(L; \mathbb{R})$ defined as follows: let $\xi$ be (a representative of) any cycle in $H_1(L; \mathbb{R})$. 
Let $Z_\xi: [0,1] \times S^1 \to M$ be the trace of $\xi$ under the isotopy. Then $\operatorname{Flux}(L_s) \in H^1(L; \mathbb{R})$ is the unique class satisfying $\langle \omega, Z_\xi \rangle = \langle \operatorname{Flux}(L_s), \xi \rangle$. It is a basic fact that the Lagrangian flux of an isotopy induced by an ambient Hamiltonian isotopy on $M$ vanishes.

One can similarly define the \emph{(symplectic) flux} of a symplectic isotopy $\{\psi_s\}$ to be the unique class $\operatorname{Flux}_{\omega}(\psi_s)$ of $H^1(M;\R)$ such that $\langle \omega, Z_\xi \rangle = \langle \operatorname{Flux}_{\omega}(\psi_s), \xi \rangle$ for all loops $\xi$ of $M$. Alternatively, it is the Lagrangian flux of the Lagrangian isotopy $\{\operatorname{graph}\psi_s\}$ in $(M\times M,\omega\oplus -\omega)$. The symplectic flux induces a homomorphism $\operatorname{Flux}_\omega:\CSymp(M)\to H^1(M;\R)$. Its image is called the \emph{flux group of $M$} and denoted by $\Gamma_\omega \subset H^1(M; \mathbb{R})$.

More generally, let $f_s: N \to M, s\in [0,1]$, be a continuous family of continuous maps and set $F:[0,1] \times N \to M, (s,n)\mapsto f_s(n)$. There is a natural map $H_1(N; \mathbb{Z}) \to \mathbb{R}$ sending a cycle $\xi$ to $\langle\omega, F([0,1] \times \xi) \rangle$ which induces a homomorphism 
$\operatorname{Flux}_{top}: \pi_1(C^0(N, M)) \to H^1(N; \mathbb{R}).$
\begin{definition}
	The image of the map $\operatorname{Flux}_{top}: \pi_1(C^0(N, M)) \to H^1(N; \mathbb{R})$ is called the \emph{topological flux group} and is denoted by $\Gamma_{top}\subset H^1(N; \mathbb{R})$.
\end{definition}

\section{Proofs}
We first record an elementary linear algebra lemma which will be used later.
\begin{lemma}\label{lemma:elementary-linear}
	Let $L$ be a finite dimensional lattice and set $V:= L \otimes_\mathbb{Z} \mathbb{R}$. If $S \subset V^*$ is a subset with the property that for every $\xi \in L \subset V$, the set 
	\begin{equation}\label{equation:finite-lattice}\{\langle s, \xi \rangle \mid s \in S\} \subset \mathbb{R}\end{equation} is countable, then $S$ is countable.  (Here $\langle -, - \rangle: V^* \times V \to \mathbb{R}$ denotes the natural pairing of $V$ with its dual.)
\end{lemma}
\begin{proof}
	Choose a basis for $V$ and let $s =(s_1, \dots, s_n) \in S$. By testing against the dual basis, we find that the coordinates $s_i$ take values in a countable set.\footnote{So in particular, it is enough to know \eqref{equation:finite-lattice} on any subset of lattice elements which form a real basis for $V$.}
\end{proof}

\subsection{The $C^\infty$ case}As a warm-up for the proof of Theorem~\ref{theoreom:main-co}, we give an alternative proof of Proposition~\ref{proposition:intro} based on the Lagrangian flux as discussed in Subsection~\ref{subsection:lag-flux}.

\begin{proof}[Second proof of Proposition~\ref{proposition:intro}]
	Suppose for contradiction that $\mathcal{P}$ is an uncountable packing. By Lemma~\ref{lemma:all-in-weinstein}, we may assume without loss of generality that $\mathcal{P}$ is entirely contained in $\mathcal{O}_{Lag}(\mathcal{U}, \tilde{\iota})$ for some Weinstein neighborhood $(\mathcal{U}, \iota)$ of $L$.
	For each $K \in \mathcal{P}$, there is a $1$-form $\alpha_K \in \Omega^1(L)$ so that $K = \operatorname{graph}(\alpha_K)$. It is enough to prove that the set of $1$-forms $\alpha_K$ representing elements of $\mathcal{P}$ is countable.
	
	The first step is to construct a piecewise smooth path of Lagrangians $(K_s)_{s \in [0,1]}$ as follows:
	\begin{enumerate}
		\item first apply the linear isotopy $[0,1/2] \ni s \mapsto \tilde{\varphi}(\op{graph}(2s \alpha_K))$
		\item then let $\{K_s\}_{s \in [1/2, 1]}$ be any Hamiltonian isotopy taking $K$ back to $L$ setwise (this exists since $K \in \mathcal{P}$)
	\end{enumerate}
	
	Now fix a curve $\xi: S^1 \to L$ representing a cycle $[\xi] \in H_1(L; \mathbb{Z})$.  Let $Z_\xi: [0,1] \times S^1 \to M$ be the cylinder swept out by the family $(K_s)$. We write $Z_\xi= Z_\xi^{(1)} \cup Z_\xi^{(2)}$, where $Z_\xi^{(i)}$ is the cylinder from step $(i)$ of the isotopy.
	
	Finally, we compute
	\begin{equation}
		\langle \operatorname{Flux}(K_s), \xi \rangle = \int_{Z_\xi} \omega = \int_{Z_\xi^{(1)}} \omega + \int_{Z_\xi^{(2)}} \omega = \int_{Z_\xi^{(1)} }\omega = \langle \alpha_K, \xi \rangle
	\end{equation}
	where $\int_{Z_\xi^{(2)}} \omega =0$ because $(K_s)_{s \in [1/2,2]}$ is a Hamiltonian isotopy.
	
	Since $Z_{\xi}$ has boundary on $L$ for all such $\xi$, it follows that $\langle [\alpha_K], \xi \rangle$ is contained in the image of the evaluation map $\omega: H_2(M, L; \mathbb{Z}) \to \mathbb{R}$, which is a countable set. By Lemma~\ref{lemma:elementary-linear}, it follows that only countably many classes $[\alpha_K] \in H^1(L;\R)$ can be represented by elements in $\mathcal{P}$. Recall however that if $[\alpha_K] = [\alpha_{K'}]$, then $K$ and $K'$ intersect. The conclusion follows.
\end{proof}

Recall that a submanifold $N \subset M$ is called \emph{instantaneously displaceable} if there is a Hamiltonian vector field of $M$ that is nowhere tangent to $N$. 
\begin{proposition} \label{prop:nonLag_subcritical}
	Let $N$ be a (closed) submanifold of $M^{2n}$ of dimension $k\leq n$. The following statements are equivalent.
	\begin{enumerate}[label=(\alph*)]
		\item $N$ admits an uncountable packing.
		\item $N$ is instantaneously displaceable.
		\item $N$ is not Lagrangian and its normal bundle admits a nowhere vanishing section.
	\end{enumerate}
\end{proposition}
We first note: 
\begin{lemma} \label{lem:uncountable-gives-section}
	If a submanifold $N$ of $M$ admits an uncountable packing by smooth isotopies, then its normal bundle has a nowhere vanishing section.
\end{lemma}
\begin{proof}
	In the $C^\infty$ topology, every element of $\Sigma^\infty(N,M)$ admits a neighbourhood such that every element in it can be represented by the graph of a section of the normal bundle of $N$ in $M$. Let $\{N_\alpha\}$ be an uncountable packing of $N$. By second countability of $\Sigma^\infty(N,M)$, there is a countable subpacking $\{N_{\alpha_i}\}_{i=1}^\infty$ such that $\bigcup_\alpha N_\alpha$ is covered by neighbourhoods as above, centered at $N_{\alpha_i}$. But then, there is some $\alpha\notin\{\alpha_i\}$ such that $N_\alpha$ is in the neighbourhood associated to some $N_{\alpha_j}$. Therefore, there is a section $\nu$ of the normal bundle of $N_{\alpha_j}$ such that $N_\alpha$ is the graph of $\nu$. Since $N_\alpha\cap N_{\alpha_j}=\emptyset$, $\nu$ does not vanish at any point.
\end{proof}

\begin{proof}[Proof of Proposition~\ref{prop:nonLag_subcritical}]
	We first show that \emph{(a)} implies \emph{(c)}. Assume that \textit{(a)} holds. By Lemma~\ref{lem:uncountable-gives-section}, $N$ has a nowhere vanishing normal vector field. Moreover, $N$ cannot be Lagrangian by Proposition~\ref{proposition:intro}.
	
	We now deduce \emph{(b)} from \emph{(c)}. Assume that \textit{(c)} holds. If $k=n$, then $N$ is instantaneously displaceable by the main result of Laudenbach--Sikorav \cite{laudenbach1994hamiltonian}. If $k<n$, this is folklore 
	(see also~\cite{gurel2008} for a formal proof of a more general result).
	
	Finally, \textit{(b)} clearly implies \textit{(a)}: if $\{\phi^t\}$ is the flow of the Hamiltonian vector field that is nowhere tangent to $N$, then $\phi^t(N)\cap\phi^s(N)=\emptyset$ for all $s,t$ small enough.
\end{proof}

\subsection{Proof of Theorem~\ref{theoreom:main-co}}
We deduce Theorem~\ref{theoreom:main-co} by combining Proposition~\ref{prop:nonLag_subcritical} with the following:
\begin{proposition}\label{proposition:lagrangian-countable_C0}
	Let $M= (M, \omega)$ be a closed symplectic manifold. Let $L \subset M$ be a Lagrangian submanifold with the property that 
	\begin{equation}\label{equation:top-assumption-body} \langle\Gamma_{top},\iota_*(H_1(L))\rangle\cap \omega(\pi_2(M)) \subset \mathbb{R}\end{equation} is discrete. The packing cardinality of $L$ is at most countable.
\end{proposition}
We will need the following lemma, whose proof is an elementary exercise using the exponential map.
\begin{lemma}\label{lemma:injectivity}
	Let $(M, g)$ be a Riemannian manifold with injectivity radius $r_0>0$. Suppose that $f_0, f_1: M \to M$ are compactly supported smooth maps with the property that $dist(f_0(x),f_1(x)) <\epsilon< r_0$ for all $x$. Then there exists a smooth homotopy $\{h^t:  M \to M\}_{t \in [0,1]}, h^0 = f_0, h^1=f_1$, with the property that $s\mapsto dist(h^s(x), h^0(x))$ is non-decreasing for all $x \in M$.
	\qed
\end{lemma}

We also need the following pieces of notation:
\begin{notation}
	If $\{\Psi^t: N \to M\}_{t \in [0,a_i]}$ is a continuous family of continuous maps, we write $\overline{\Psi}^t:= \Psi^{1-t}$.
	
	If $\{\Phi_i^t: N \to M\}_{t \in [0,a_i]}$ is a continuous family of continuous maps with $\Phi_1^{a_1} = \Phi_2^0$, we let $\{ (\Phi_1 \# \Phi_2)^s\}_{s \in [0,a_1+a_2]}$ be the ``concatenation from left to right'' of these families. In other words, $(\Phi_1 \# \Phi_2)^s(-) = \Phi_1^s(-)$ for $s \in [0,a_1]$ and  $(\Phi_1 \# \Phi_2)^s(-) = \Phi_2^{1-s}(-)$ for $s \in [a_1,a_1+a_2]$. 
\end{notation}
	
We now begin the proof of Proposition~\ref{proposition:lagrangian-countable_C0}. Arguing as in the second proof of Proposition~\ref{proposition:intro}, let us suppose again for contradiction that $\mathcal{P}$ is an uncountable packing. We may again assume without loss of generality that $\mathcal{P}$ is entirely contained in $\mathcal{O}_{Lag}(\mathcal{U}, \tilde{\iota})$ for some Weinstein neighborhood $(\mathcal{U}, \iota)$ of $L$. For each $K \in \mathcal{P}$, there is a $1$-form $\alpha_K \in \Omega^1(L)$ so that $K = \tilde{\iota}(\operatorname{graph}(\alpha_K))$. It is enough to prove that the set of $1$-forms $\alpha_K$ representing elements of $\mathcal{P}$ is countable.

Fix an auxiliary compatible complex structure on $M$; henceforth we will always measure distances with respect to the induced metric. For $r \geq 1$, let $\mathcal{V}$ be a Weinstein neighborhood of $K$ of radius $1/r$. By choosing $r$ large enough, we can assume $\mathcal{V} \subset \mathcal{U}$. 

Let $\phi$ be a compactly-supported Hamiltonian homeomorphism of $M$ such that $\phi(L)=K$, and let $\{\phi_k\}$ be a sequence of compactly-supported Hamiltonian diffeomorphisms that $C^0$-converges to $\phi$. 
Without loss of generality, we may assume that $\phi_k(L) \subset \mathcal{V}$. Finally, we fix for each $k$ a compactly-supported Hamiltonian isotopy $\{\phi^t_k\}_{t\in [0,1]}$ from the identity to $\phi^1_k=\phi_k$.

By Lemma~\ref{lemma:injectivity}, we may assume after possibly forgetting finitely many terms in the sequence $\{\phi_k\}$ that there exists a homotopy $$\{f_k^t: M \to M\}_{t \in [0,1]}$$ such that $f_k^0=\phi_k, f_k^1= \phi$. We can further assume that $f_k^t(L) \subset \mathcal{V}$ for all $t \in [0,1]$ and that $\operatorname{dist}(f_k^{t_0}(x), f_k^{t_1}(x))< 1/r$ for all $t_1, t_2 \in [0,1]$.

\subsubsection{Construction}
	Fix a loop $\xi: S^1 \to L$. We construct a closed cycle $Z_{k, \xi}$ by concatenating three cylinders,\footnote{The terms loop/cylinder always refer to \emph{maps} with domain $S^1\times S^1$ or $[0,1]\times S^1$, which need not be embeddings. The positive (resp.\ negative) boundary of a cylinder $[0,1] \times S^1 \to M$ is understood to be the restriction of the map to $\{1\} \times S^1$ (resp.\ $\{0\} \times S^1$.} as follows:
	
	\begin{enumerate}
		\item Let $C_\xi$ be the cylinder swept by $\xi$ through the Lagrangian isotopy $t\mapsto \tilde{\iota}(\operatorname{graph}(t\alpha_K))$. Let $\ell_C^+$ resp.\ $\ell_C^-$ be the positive resp.\ negative boundary components; by construction, we have $\ell_C^-= \xi$.
		\item Let $D_{k, \xi}$ be the cylinder swept by $\ell_C^+$ through the homotopy $\{\ov{f}^t_k\}$. By construction, $D_{k, \xi}$ is contained in $\mathcal{V}$.  We let $\ell_{D,k}^\pm$ be the positive/negative boundary components; by construction $\ell_{D,k}^-= \ell_{C,k}^+$.
		
		\item Let $E_{k, \xi}$ be the cylinder swept by $\ell_{D,k}^+$ by the Hamiltonian isotopy $\{\ov{\phi}_k\}$. We let $\ell_{E,k}^\pm$ be the positive/negative boundary components; by construction $\ell_{E,k}^-= \ell_{D,k}^+$ and $\ell_{E, k}^+= \phi_k^{-1} (\phi_k(\xi))= \xi$.
	\end{enumerate}
	We let $Z_{k, \xi} ``:=" C_\xi \# D_{\xi, k}\# E_{\xi, k}$ be the cyclic concatenation of $C_\xi, D_{\xi, k}, E_{\xi, k}$ along their common boundary components $\ell_C^+= \ell^-_{D, k}, \ell^+_{D, k}= \ell^-_{E, k}, \ell^+_{E, k}= \xi = \ell_C^-$.

\begin{lemma}\label{lemma:limit-zk}
	We have $\langle [\alpha_K], \xi \rangle=\lim_k \omega(Z_{k,\xi})$. 
\end{lemma}
\begin{proof}
	Note that
	\begin{align} \label{eq:cylinders_flux}
		\langle [\alpha_K], \xi \rangle=\omega(C_\xi)=\omega(Z_{k,\xi})-\omega(D_{k,\xi})-\omega(E_{k,\xi}).
	\end{align}
	But note that $E_{k,\xi}$ is swept out by a Hamiltonian isotopy and thus $\omega(E_{k,\xi}) = 0$.
	Finally, by \cite[Thm.\ 2]{m-o}, we have that
	\begin{align*}
		|\omega(D_{k,\xi})|\leq r\ell^{\min}_g(\pi_*\ell^+_{D,k})=r\ell^{\min}_g(\ell^-_{D,k}),
	\end{align*}
	where $\ell^{\min}_g(\beta)$ is the minimal length of a geodesic loop in $(K,g)$ representing $\beta$ and $\pi:\phi_k(L)\to K$ is the projection induced by the inclusion of $\phi_k(L)$ into the Weinstein neighbourhood $\mathcal{V}$ of $K$. But, as $k$ tends to infinity, we may take $r$ above tending to zero, so that $\omega(D_{k,\xi})\to 0$. Therefore, $\langle [\alpha_K], \xi \rangle=\lim \omega(Z_{k,\xi})$. 
\end{proof}

\begin{lemma}\label{lemma:h1-discrete}
	We have $$\omega(Z_{k,\xi})-\omega(Z_{\ell,\xi}) \in \langle\Gamma_{top},H_1(L)\rangle.$$ 
\end{lemma}
\begin{proof} Note that
	\begin{equation*}
		\omega(Z_{k,\xi})-\omega(Z_{\ell,\xi})=\omega(\overline{Z_{\ell,\xi}}\#Z_{k,\xi})=\omega(\overline{E_{\ell,\xi}}\#\overline{D_{\ell,\xi}}\#D_{k,\xi}\#E_{k,\xi}),
	\end{equation*}
	where $\#$ denotes concatenation and $\overline{\cdot}$ the reversal of orientation. By construction, $\overline{E_{\ell,\xi}}\#\overline{D_{\ell,\xi}}\#D_{k,\xi}\#E_{k,\xi}$ is the torus swept out by $\xi$ under the concatenation of smooth maps 
	$\{\phi_\ell\#f_\ell\# \ov{f}_k\# \ov{\phi}_k\}$.
\end{proof}

\begin{lemma}\label{lemma:omega-discrete}
	We have $$\omega(Z_{k,\xi})-\omega(Z_{\ell,\xi}) \in \omega(\pi_2(M)).$$ 
\end{lemma}
\begin{proof}
	The argument is inspired by the proof of \cite[Prop.\ 15]{a-c-l-s}. Observe that it suffices to show that there is some point $x$ in the image the loop $\xi$ such that the loop $t \mapsto (\phi_\ell\# f_\ell \#\ov{f}_k\#\ov{\phi}_k)^t(x)$ is contractible. 
	
	Now, since $\{\phi^t_\ell\}$ is Hamiltonian, there exists some point $y \in M$ (possibly far away from $L$) such that $[0,1] \ni t \mapsto \phi^t_\ell(y)$ is a contractible loop: this follows from the (now-classical) well-definedness of Floer homology for non-degenerate Hamiltonians on closed symplectic manifolds.\footnote{If $M$ is non-compact, then by construction $\{\phi^t_\ell\}$ is compactly-supported, so the conclusion is obvious.} 
	
	Without loss of generality, $M$ is connected. Hence we can joint $x$ and $y$ by a path $\sigma: [0,1] \to M$ and $\sigma(0)=x, \sigma(1)=y$. But now the family
	$$(\phi_\ell \# f_\ell \# \ov{f}_k \# \ov{\phi}_k)^t (\sigma(s))$$
	defines a free homotopy from the loop 
	$$t \mapsto (\phi_\ell\#f_\ell\#\ov{f}_k \#\ov{\phi}_k)^t(x)$$ to the loop $$t \mapsto (\phi_\ell\#f_\ell\#\ov{f}_k\#\ov{\phi}_k)^t(y)= (f_\ell\#\ov{f}_k)^t(y).$$
	But by construction of the $f^t_\ell$, the loop $(f_\ell\#\ov{f}_k)(z)$ must be fully contained in a geodesic ball centered at $z=\phi_\ell(y)$. Therefore, this last loop~--~and thus the original one~---~is contractible.  
\end{proof}

\begin{proof}[Proof of Proposition~\ref{proposition:lagrangian-countable_C0}]
	Choose a class $H^1(L, \mathbb{Z})$ and let $\xi$ be a loop on $L$ representing it. By combining Lemmata~\ref{lemma:limit-zk}, \ref{lemma:h1-discrete}, and~\ref{lemma:omega-discrete}, we have $\langle [\alpha_K], \xi \rangle = \omega(Z_{k, \xi})$ for all $k$ large enough. 
	
	Hence $$\langle [\alpha_K], \xi \rangle \in \{\operatorname{im}( \omega(-): H_2(M, L; \mathbb{Z}) \to \mathbb{R} \} \subset \mathbb{R},$$ which is manifestly countable. So the conclusion follows from Lemma~\ref{lemma:elementary-linear}.
\end{proof}

\begin{remark}
	Tracing through the above argument, the only place where we used the assumption that \eqref{equation:top-assumption-body} is discrete was to ensure that $\omega(Z_{k,\xi})$ is eventually independent of $k$. Instead, it would also be enough to assume that $\phi$ can be ``well approximated'' by Hamiltonian diffeomorphisms, i.e.\ the approximating sequence $\{\phi_k\}$ may be chosen so that the relative homotopy class of the path $\phi_k\# f_k$ in $C^0(M,M)$ is constant.
	
	This can always be done if, on a $C^0$-neighbourhood of the identity in $\Ham(M)$, every Hamiltonian diffeomorphism is the time-1 map of a Hamiltonian isotopy $\{\phi^t_H\}$ such that $\max_t d_{C^0}(id,\phi^t_H)\leq Cd_{C^0}(id,\phi^1_H)^\alpha$, for some $C,\alpha>0$ independent of the isotopy. In general, this is a very hard property to prove, but it is known to hold when $M$ is a closed surface or the Euclidean ball (see Remark~3.4 and Lemma~3.2 of~\cite{seyfaddini2013}, respectively). However, in those examples, $\omega(\pi_2(M))$ is discrete so Proposition~\ref{proposition:lagrangian-countable_C0} already does the trick as currently stated.
\end{remark}

\section{Further questions}
\begin{enumerate}
	\item A Legendrian variant: let $(V, \xi)$ be a contact manifold. \emph{Can one find uncountably many closed Legendrian submanifolds which are pairwise Legendrian isotopic, no two of which are connected by a Reeb chord?}
	\item Let $(M^{2n}, \omega)$ be symplectic and let $K \subset M, dim(K) \geq n$ be a submanifold whose normal bundle has a nowhere vanishing section. Prove or disprove: \emph{$K$ admits an uncountable packing if and only if it $K$ is \emph{not} coisotropic.}
	
	The case $dim(K)=n$ is handled by Proposition~\ref{proposition:intro} and Proposition~\ref{prop:nonLag_subcritical}. See Gürel \cite{gurel2008} for a partial result when $dim(K)>n$.
	
	\item  Other variants of Definition~\ref{definition:packing} are certainly possible, in fact arguably more natural. For example, one could define a $C^0$-packing as any collection of \emph{arbitrary} subsets $\{\Sigma_\alpha\}$ such that $\Sigma_\alpha = \phi_\alpha(\Sigma)$ for some $\phi_\alpha \in \ov{\Ham}(M, \omega)$. 
\end{enumerate}

\section*{Acknowledgements}
We thanks Georgios Dimitroglou Rizell and Leonid Polterovich for helpful communications about the status of the questions treated in this paper. We also thank the anonymous referee for several helpful comments.

The discussions which led to this note began when the third author visited ETH Zürich in February 2025; he thanks the Department of Mathematics for its wonderful hospitality. 

\end{document}